\documentclass[11pt,reqno]{amsart}

\usepackage{amssymb, amsmath, amsthm}
\usepackage{hyperref}
\usepackage[alphabetic,lite]{amsrefs}
\usepackage{verbatim}
\usepackage{amscd}   % for commutative diagrams
\usepackage[all]{xy} % for complicated commutative diagrams
\usepackage{youngtab} % for Young tableaux
\usepackage{young} % for Young tableaux
\usepackage{tikz}
\usepackage{ mathrsfs }
\usepackage{cases}

\newcommand{\defi}[1]{{\upshape\sffamily #1}}

\renewcommand{\a}{\alpha}

\newcommand{\bul}{\bullet}
\newcommand{\complex}{\mathbb{C}}
\renewcommand{\d}{\delta}
\newcommand{\D}{\mathcal{D}}
\newcommand{\ds}{\displaystyle\sum}
\newcommand{\bw}{\bigwedge}

\renewcommand{\ll}{\lambda}

\newcommand{\oo}{\otimes}
\newcommand{\p}{\partial}

\newcommand{\s}{\sigma}

\newcommand{\coldet}{\operatorname{col-det}}
\renewcommand{\det}{\operatorname{det}}
\newcommand{\diag}{\operatorname{diag}}
\newcommand{\dom}{\operatorname{dom}}

\newcommand{\rank}{\operatorname{rank}}

\newcommand{\sgn}{\operatorname{sgn}}

\newcommand{\GL}{\operatorname{GL}}

\newcommand{\Id}{\operatorname{Id}}

\newcommand{\Pf}{\operatorname{Pf}}
\newcommand{\SL}{\operatorname{SL}}

\newcommand{\Sym}{\operatorname{Sym}}

\newcommand{\bb}[1]{\mathbb{#1}}

\renewcommand{\rm}[1]{\textrm{#1}}
\newcommand{\mc}[1]{\mathcal{#1}}
\newcommand{\mf}[1]{\mathfrak{#1}}

\newcommand{\ul}[1]{\underline{#1}}
\newcommand{\scpr}[2]{\left\langle #1,#2 \right\rangle}

\def\lra{\longrightarrow}

\newtheorem{theorem}{Theorem}[section]
\newtheorem*{theorem*}{Theorem}
\newtheorem{lemma}[theorem]{Lemma}
\newtheorem{conjecture}[theorem]{Conjecture}
\newtheorem{proposition}[theorem]{Proposition}
\newtheorem{corollary}[theorem]{Corollary}
\newtheorem*{corollary*}{Corollary}

\newtheorem*{minors*}{Theorem on Maximal Minors}
\newtheorem*{Pfaffians*}{Theorem on sub--maximal Pfaffians}

\theoremstyle{definition}
\newtheorem{definition}[theorem]{Definition}
\newtheorem*{definition*}{Definition}

\theoremstyle{remark}
\newtheorem{remark}[theorem]{Remark}
\newtheorem*{remark*}{Remark}

\numberwithin{equation}{section}

\begin{document}

\title{Bernstein--Sato polynomials for maximal minors and sub--maximal Pfaffians}

\author{Andr\'as C. L\H{o}rincz}
\address{Department of Mathematics, University of Connecticut, Storrs, CT 06269}
\email{andras.lorincz@uconn.edu}

\author{Claudiu Raicu}
\address{Department of Mathematics, University of Notre Dame, Notre Dame, IN 46556\newline
\indent Institute of Mathematics ``Simion Stoilow'' of the Romanian Academy}
\email{craicu@nd.edu}

\author{Uli Walther}
\address{Department of Mathematics, Purdue University, West Lafayette, IN 47907}
\email{walther@math.purdue.edu}

\author{Jerzy Weyman}
\address{Department of Mathematics, University of Connecticut, Storrs, CT 06269}
\email{jerzy.weyman@uconn.edu}


\date{\today}


\begin{abstract} We determine the Bernstein-Sato polynomials for the ideal of maximal minors of a generic $m\times n$ matrix, as well as for that of sub-maximal Pfaffians of a generic skew-symmetric matrix of odd size. As a corollary, we obtain that the Strong Monodromy Conjecture holds in these two cases.
\end{abstract}

\maketitle

\section{Introduction}\label{sec:intro}

Consider a polynomial ring $S=\bb{C}[x_1,\cdots,x_N]$ and let $\D=S[\p_1,\cdots,\p_N]$ denote the associated \defi{Weyl algebra} of differential operators with polynomial coefficients ($\p_i=\frac{\p}{\p x_i}$). For a non-zero element $f\in S$, the set of polynomials $b(s)\in\bb{C}[s]$ for which there exists a differential operator $P_b\in\D[s]$ such that
\begin{equation}\label{eq:opPb}
P_b\cdot f^{s+1}=b(s)\cdot f^s
\end{equation}
form a non-zero ideal. The monic generator of this ideal is called the \defi{Bernstein--Sato polynomial (or the $b$-function) of $f$}, and is denoted $b_f(s)$. The $b$-function gives a measure of the singularities of the scheme defined by $f=0$, and its zeros are closely related to the eigenvalues of the monodromy on the cohomology of the Milnor fiber. In the case of a single hypersurface, its study has originated in \cites{bernstein,sato-shintani}, and later it has been extended to more general schemes in \cite{budur-mustata-saito} (see Section~\ref{subsec:bsatoideal}). Despite much research, the calculation of $b$-functions remains notoriously difficult: several algorithms have been implemented to compute $b$-functions, and a number of examples have been worked out in the literature, but basic instances such as the $b$-functions for determinantal varieties are still not understood. In \cite{budur-bsatogenl} and \cite{budur-barcelona}, Budur posed as a challenge and reviewed the progress on the problem of computing the $b$-function of the ideal of $p\times p$ minors of the generic $m\times n$ matrix. We solve the challenge for the case of maximal minors in this paper, and we also find the $b$-function for the ideal of $2n\times 2n$ Pfaffians of the generic skew-symmetric matrix of size $(2n+1)\times(2n+1)$. For maximal minors, our main result is as follows:

\begin{minors*}[Theorem~\ref{thm:bfunmaxlminors}]
 Let $m\geq n$ be positive integers, consider the generic $m\times n$ matrix of indeterminates $(x_{ij})$, and let $I=I_n$ denote the ideal in the polynomial ring $S=\bb{C}[x_{ij}]$ which is generated by the $n\times n$ minors of $(x_{ij})$. The $b$-function of $I$ is given by
 \[b_I(s)=\prod_{i=m-n+1}^m (s+i).\]
\end{minors*}

When $m=n$, $I$ is generated by a single equation -- the determinant of the generic $n\times n$ matrix -- and the formula for $b_I(s)$ is well-known (see \cite[Appendix]{prehomogeneous} or \cite[Section~5]{raicu-Dmods}). For general $m\geq n$, if we let $Z_{m,n}$ denote the zero locus of $I$, i.e. the variety of $m\times n$ matrices of rank at most $n-1$, then using the renormalization (\ref{eq:defbZ}) our theorem states that the $b$-function of $Z_{m,n}$ is $\prod_{i=0}^{n-1}(s+i)$. It is interesting to note that this only depends on the value of $n$ and not on $m$.

The statement of the Strong Monodromy Conjecture of Denef and Loeser \cite{denef-loeser} extends naturally from the case of one hypersurface to arbitrary ideals, and it asserts that the poles of the topological zeta function of $I$ are roots of $b_I(s)$. We verify this conjecture for maximal minors and sub-maximal Pfaffians in Section~\ref{sec:SMC}. When $I=I_n$ is the ideal of maximal minors of $(x_{ij})$, the methods of \cite{docampo} can be used to show that the set of poles of the topological zeta function of $I$ is $\{-m,-m+1,\cdots,-m+n-1\}$, and therefore it coincides precisely with the set of roots of $b_I(s)$. If we replace $I$ by the ideal $I_p$ of $p\times p$ minors of $(x_{ij})$, $1<p<n$, then this is no longer true: as explained in \cite[Example~2.12]{budur-barcelona}, a computer calculation of T. Oaku shows that for $m=n=3$ one has $b_{I_2}(s)=(s+9/2)(s+4)(s+5)$, while \cite[Thm.~6.5]{docampo} shows that the only poles of the zeta function of $I_2$ are $-9/2$ and $-4$. Besides the Strong Monodromy Conjecture which predicts some of the roots of $b_{I_p}(s)$, we are not aware of any general conjectural formulas for $b_{I_p}(s)$ when $1<p<n$.

In the case of Pfaffians we prove:

\begin{Pfaffians*}[Theorem~\ref{thm:bfunskew}]
 Let $n$ be a positive integer, and with the conventions $x_{ii}=0$, $x_{ij}=-x_{ji}$, consider the generic $(2n+1)\times (2n+1)$ generic skew-symmetric matrix of indeterminates $(x_{ij})$. If we let $I$ denote the ideal in the polynomial ring $S=\bb{C}[x_{ij}]$ which is generated by the $2n\times 2n$ Pfaffians of $(x_{ij})$ then the $b$-function of $I$ is given by
 \[b_I(s)=\prod_{i=0}^{n-1}(s+2i+3).\]
\end{Pfaffians*}

If we write $Z_n$ for the zero locus of $I$, i.e. the variety of $(2n+1)\times(2n+1)$ skew-symmetric matrices of rank at most $(2n-2)$, then by (\ref{eq:defbZ}) we get $b_{Z_n}(s)=\prod_{i=0}^{n-1}(s+2i)$. By \cite[Appendix]{prehomogeneous} or \cite[Section~6]{raicu-Dmods}, this is the same as the $b$-function of the hypersurface of singular $2n\times 2n$ skew-symmetric matrices.

\subsection*{Organization}

In Section~\ref{sec:preliminaries} we review some generalities on representation theory and $\D$-modules, we recall the necessary results on invariant differential operators and their eigenvalues, and we state the basic results and definitions regarding $b$-functions of arbitrary ideals. In Section~\ref{sec:boundingbfun} we illustrate some methods for bounding the $b$-function of an ideal: for upper bounds we use invariant differential operators, while for lower bounds we show how non-vanishing of local cohomology can be used to exhibit roots of the $b$-functions. These methods allow us to compute the $b$-function for sub-maximal Pfaffians, and to bound from above the $b$-function for maximal minors. In Section~\ref{sec:bsatomaxlminors} we employ the $\SL_n$-symmetry in the definition of the $b$-function of maximal minors in order to show that the upper bound obtained in Section~\ref{sec:boundingbfun} is in fact sharp. In Section~\ref{sec:SMC} we give a quick derivation, based on our main results, of the Strong Monodromy Conjecture for maximal minors and sub-maximal Pfaffians.

\subsection*{Notation and conventions}

We write $[N]$ for the set $\{1,\cdots,N\}$, and for $k\leq N$ we let ${[N]\choose k}$ denote the collection of $k$-element subsets of $[N]$. Throughout the paper, $X=\bb{A}^N$ is an affine space, and $S=\bb{C}[x_1,\cdots,x_N]$ denotes the coordinate ring of $X$. We write $\D_X$ or simply $\D$ for the \defi{Weyl algebra} of differential operators on $X$: $\D_X=S[\p_1,\cdots,\p_n]$ where $\p_i=\frac{\p}{\p x_i}$. In order to distinguish between the various kinds of tuples that arise in this paper, we will try as much as possible to stick to the following conventions: we write
\begin{itemize}
 \item $f=(f_1,\cdots,f_r)\in S^r$ for a tuple of polynomials in $S$.
 \item $\hat{c}=(c_1,\cdots,c_r)\in\bb{Z}^r$ for a tuple of integers indexing the operators in the definition of $b$-functions.
 \item $\ul{s}=(s_1,\cdots,s_r)$ for a tuple of independent variables used to define $b$-functions.
 \item $\ul{\a}=(\a_1,\cdots,\a_r)\in\bb{Z}^r$ when $\a_i$ are exponents which arise as specializations of the variables $s_i$.
 \item $\ll=(\ll_1,\cdots,\ll_r)\in\bb{Z}^r$ for a dominant weight or partition.
\end{itemize}

\section{Preliminaries}\label{sec:preliminaries}

\subsection{Representation Theory}\label{subsec:repthy}
We consider the group $\GL_N=\GL_N(\bb{C})$ of invertible $N\times N$ complex matrices, and denote by $T_N$ the maximal torus of diagonal matrices. We will refer to $N$--tuples $\ll=(\ll_1,\cdots,\ll_N)\in\bb{Z}^N$ as \defi{weights} of $T_N$ and write $|\ll|$ for the total size $\ll_1+\cdots+\ll_N$ of $\ll$. We say that $\ll$ is a \defi{dominant weight} if $\ll_1\geq\ll_2\geq\cdots\geq\ll_N$ and denote the collection of dominant weights by~$\bb{Z}^N_{\dom}$. A dominant weight with $\ll_N\geq 0$ is a \defi{partition}, and we write $\mc{P}^N$ for the set of partitions in $\bb{Z}^N_{\dom}$. We will implicitly identify $\mc{P}^{N-1}$ with a subset of $\mc{P}^N$ by setting $\ll_N=0$ for any $\ll\in\mc{P}^{N-1}$. For $0\leq k\leq N$ and $a\geq 0$ we write $(a^k)$ for the partition $\ll\in\mc{P}^k\subset\mc{P}^N$ with $\ll_1=\cdots=\ll_k=a$. Irreducible rational representations of $\GL_N(\bb{C})$ are in one-to-one correspondence with dominant weights $\ll$. We denote by $S_{\ll}\bb{C}^N$ the irreducible representation associated to~$\ll$, often referred to as a \defi{Schur functor}, and note that $S_{(1^k)}\bb{C}^N=\bw^k\bb{C}^N$ is the $k$-th exterior power of $\bb{C}^N$ for every $0\leq k\leq N$. When $m\geq n$, we have \defi{Cauchy's formula} \cite[Cor.~2.3.3]{weyman}

\begin{equation}\label{eq:Cauchy}
\Sym(\bb{C}^m\oo\bb{C}^n)=\bigoplus_{\ll\in\mc{P}^n} S_{\ll}\bb{C}^m\oo S_{\ll}\bb{C}^n.
\end{equation}
If we identify $\bb{C}^m\oo\bb{C}^n$ with the linear forms on the space $X=X_{m\times n}$ of $m\times n$ complex matrices, then (\ref{eq:Cauchy}) is precisely the decomposition into irreducible $\GL_m\times\GL_n$ representations of the coordinate ring of~$X$. For a partition $\ll$ we write $\ll^{(2)}=(\ll_1,\ll_1,\ll_2,\ll_2,\cdots)$ for the partition obtained by repeating each part of $\ll$ twice. The skew-symmetric version of Cauchy's formula \cite[Prop.~2.3.8(b)]{weyman} yields
\begin{equation}\label{eq:skewCauchy}
\Sym\left(\bw^2\bb{C}^{2n+1}\right)=\bigoplus_{\ll\in\mc{P}^n} S_{\ll^{(2)}}\bb{C}^{2n+1}.
\end{equation}
If we identify $\bw^2\bb{C}^{2n+1}$ with the linear forms on the space $X=X_n$ of $(2n+1)\times(2n+1)$ skew-symmetric matrices, then (\ref{eq:skewCauchy}) describes the decomposition into irreducible $\GL_{2n+1}$-representations of the coordinate ring of $X$.

\subsection{Invariant operators and $\D$-modules}\label{subsec:invDmods}

Throughout this paper we will be studying various (left) $\D_X$-modules when $X$ is a finite dimensional representation of some connected reductive linear algebraic group~$G$. Differentiating the $G$-action on $X$ yields a map from the Lie algebra $\mf{g}$ into the vector fields on $X$, which in turn induces a map
\begin{equation}\label{eq:deftau}
 \tau:\mc{U}(\mf{g})\lra\D_X,
\end{equation}
where $\mc{U}(\mf{g})$ denotes the universal enveloping algebra of $\mf{g}$. In particular, any $\D_X$-module $\mc{M}$ inherits via $\tau$ the structure of a $\mf{g}$-representation: if $g\in\mf{g}$ and $m\in\mc{M}$ then $g\cdot m=\tau(g)\cdot m$. In order to make the action of $\D_X$ on $\mc{M}$ compatible with the $\mf{g}$-action we need to consider the action of $\mf{g}$ on $\D_X$ given by
\begin{equation}\label{eq:bulactiong}
g\bul p = \tau(g)\cdot p - p\cdot\tau(g)\rm{ for }g\in\mf{g}\rm{ and }p\in\D_X.
\end{equation}
The induced Lie algebra action of $\mf{g}$ on the tensor product $\D_X\oo\mc{M}$ makes the multiplication $\D_X\oo\mc{M}\to\mc{M}$ into a $\mf{g}$-equivariant map: $g\cdot(p\cdot m)=(g\bul p)\cdot m+p\cdot(g\cdot m)$ for $g\in\mf{g}$, $p\in\D_X$, $m\in\mc{M}$.

We also use the symbol $\bul$ to avoid a possible source of confusion that may arise as follows. Since $S$ is both a $\D_X$-module and a subset of $\D_X$, the multiplication of an element $p\in\D_X$ with an element $f\in S$ can have two meanings: we write $p\bul f$ for the result of applying the operator $p$ to $f$, and $p\cdot f$ for the multiplication of $p$ with $f$ inside $\D_X$. The operation $p\bul f$ is only used twice in our paper: to discuss the pairing between differential operators and polynomials (see (\ref{eq:dualbases})), and in Section~\ref{subsec:bsatoideal} when we refer to $\p_i\bul f_j$, the $i$-th partial derivative of $f_j$.

For a Lie subalgebra $\mf{a}\subset\mf{g}$, and a $\D_X$-module $\mc{M}$, we consider the collection $\mc{M}^{\mf{a}}$ of \defi{$\mf{a}$-invariant sections in~$\mc{M}$}:
\[\mc{M}^{\mf{a}}=\{m\in\mc{M}:\tau(a)\cdot m=0\rm{ for all }a\in\mf{a}\}.\]
The main examples that we study arise from a tuple $f=(f_1,\cdots,f_r)\in S^r$ of polynomial functions on $X$, where each $f_i$ is $\mf{a}$-invariant, and $\mc{M}=S_{f_1\cdots f_r}$ is the localization of $S$ at the product $f_1\cdots f_r$. In this case we have that $\mc{M}^{\mf{a}}=(S_{f_1\cdots f_r})^{\mf{a}}$ coincides with $(S^{\mf{a}})_{f_1\cdots f_r}$, the localization of $S^{\mf{a}}$ at $f_1\cdots f_r$.

The \defi{ring of $\mf{a}$-invariant differential operators on $X$}, denoted by $\D_X^{\mf{a}}$ (not to be confused with $\mc{M}^{\mf{a}}$ for $\mc{M}=\D_X$ as defined above), are defined via
\begin{equation}\label{eq:defDXs}
\D_X^{\mf{a}}=\{p\in\D_X:a\bul p=0\rm{ for all }a\in\mf{a}\}, 
\end{equation}
and $\mc{M}^{\mf{a}}$ is a $\D_X^{\mf{a}}$-module whenever $\mc{M}$ is a $\D_X$-module. If we write $\mc{ZU}(\mf{a})$ for the center of $\mc{U}(\mf{a})$ then it follows from (\ref{eq:bulactiong}) and (\ref{eq:defDXs}) that
\begin{equation}\label{eq:tauZUgsubD^g}
 \tau\left(\mc{ZU}(\mf{a})\right)\subseteq\D_X^{\mf{a}}.
\end{equation}
An alternative way of producing $\mf{a}$-invariant differential operators is as follows. Let $P=\bb{C}[\p_1,\cdots,\p_N]$ and write $S_k$ (resp. $P_k$) for the subspace of $S$ (resp. $P$) of homogeneous elements of degree $k$. The action of $P$ on $S$ by differentiation induces $\mf{a}$-equivariant perfect pairings $\scpr{\ }{\ }:P_k\times S_k\to\bb{C}$ for each $k\geq 0$, namely $\scpr{w}{v}=w\bul v$. If $V\subset S_k$, $W\subset P_k$ are dual $\mf{a}$-subrepresentations, with (almost dual) bases $v=(v_1,\cdots,v_t)$ and $w=(w_1,\cdots,w_t)$, such that for some non-zero constant $c$
\begin{equation}\label{eq:dualbases}
 \scpr{w_i}{v_j}=0\rm{ for }i\neq j,\quad\scpr{w_i}{v_i}=c\rm{ for all }i,
\end{equation}
then we can define elements of $\D_X^{\mf{a}}$ via
\begin{equation}\label{eq:defDp}
D_{v,w}=\sum_{i=1}^t v_i\cdot w_i,\quad D_{w,v}=\sum_{i=1}^t w_i\cdot v_i.
\end{equation}
In the examples that we consider, the basis $w$ will have a very simple description in terms of $v$. For $p=p(x_1,\cdots,x_N)\in S$, we define $p^*=p(\p_1,\cdots,\p_N)\in P$. For the tuples of maximal minors and sub-maximal Pfaffians, it will suffice to take $w_i=v_i^*$ in order for (\ref{eq:dualbases}) to be satisfied, in which case we'll simply write $D_v$ instead of $D_{w,v}$ and $D_{v^*}$ or $D_w$ instead of $D_{v,w}$.

We specialize our discussion to the case when $X=X_{m,n}$ is the vector space of $m\times n$ matrices, $m\geq n$, and $G=\GL_m\times\GL_n$, $\mf{g}=\mf{gl}_m\oplus\mf{gl}_n$. The coordinate ring of $X$ is $S=\bb{C}[x_{ij}]$ with $i\in[m]$, $j\in[n]$. We consider the tuple of maximal minors $d=(d_K)_{K\in{[m]\choose n}}$ of the generic matrix of indeterminates $(x_{ij})$, where
\begin{equation}\label{eq:defdK}
d_K=\det(x_{ij})_{i\in K,j\in[n]},
\end{equation}
and the tuple $\p=(\p_K)_{K\in{[m]\choose n}}$ of maximal minors in the dual variables
\begin{equation}\label{eq:defpK}
\p_K=d_K^*=\det(\p_{ij})_{i\in K,j\in[n]},
\end{equation}
The elements $d_K$ form a basis for the irreducible representation $V=\bw^n\bb{C}^m\oo\bw^n\bb{C}^n$ in (\ref{eq:Cauchy}), indexed by the partition $\ll=(1^n)$, while $\p_K$ form a basis for the dual representation $W$. If we let $c=n!$ then it follows from Cayley's identity \cite[(1.1)]{css} that (\ref{eq:dualbases}) holds for the tuples $d$ and $\p$, so we get $\mf{g}$-invariant operators
\begin{equation}\label{eq:defDd}
 D_{\p}=\sum_{K\in{[m]\choose n}}d_K\cdot\p_K,\quad D_{d}=\sum_{K\in{[m]\choose n}}\p_K\cdot d_K.
\end{equation}

If we consider $\mf{sl}_n\subset\mf{gl}_n\subset\mf{g}$, the special linear Lie algebra of $n\times n$ matrices with trace $0$, then
\[S^{\mf{sl}_n}=\bb{C}\left[d_K:K\in{[m]\choose n}\right]\]
is the $\bb{C}$-subalgebra of $S$ generated by the maximal minors $d_K$. Moreover, $S^{\mf{sl}_n}$ can be identified with the homogeneous coordinate ring of the Grassmannian $\bb{G}(n,m)$ of $n$-planes in $\bb{C}^m$. We let
\begin{equation}\label{eq:defpij}
 p_0=d_{[n]},\rm{ and }p_{ij}=d_{[n]\backslash\{i\}\cup\{j\}},\rm{ for }i\in [n], j\in [m]\setminus [n],
\end{equation}
and note that $p_{ij}/p_0$ give the coordinates on the open Schubert cell defined by $p_0\neq 0$ inside $\bb{G}(n,m)$. It follows that if we take any $K\in{[m]\choose n}$, set $|[n]\setminus K|=k$, and enumerate the elements of the sets $[n]\backslash K=\{i_1,\cdots, i_k\}$ and $K\backslash [n]=\{j_1,\cdots,j_k\}$ in increasing order, then:
\begin{equation}\label{eq:rel}
d_K=p_0^{1-k}\cdot\det \begin{pmatrix}
p_{i_1j_1} & \cdots & p_{i_1j_k}\\
\vdots & \ddots & \vdots\\
p_{i_kj_1} & \cdots & p_{i_kj_k}
\end{pmatrix}.
\end{equation}
It will be important in Section~\ref{sec:bsatomaxlminors} to note moreover that $p_0,p_{ij}$ are algebraically independent and that 
\begin{equation}\label{eq:slninvariantsS_p}
 \left\{p_0^{c_0}\cdot\prod_{i\in[n],j\in[m]\setminus[n]}p_{ij}^{c_{ij}}:c_0,c_{ij}\in\bb{Z}\right\}\rm{ forms a }\bb{C}\rm{-basis of }\left(S_{p_0\cdot\prod_{i,j}p_{ij}}\right)^{\mf{sl}_n}.
\end{equation}

\subsection{Capelli elements, eigenvalues, and the Fourier transform {\cite{howe-umeda}}}\label{subsec:CapelliFourier}

Throughout the paper, by the determinant of a matrix $A=(a_{ij})_{i,j\in[r]}$ with non-commuting entries we mean the \defi{column-determinant}: if $\mc{S}_r$ is the symmetric group of permutations of $[r]$, and $\sgn$ denotes the signature of a permutation, then
\begin{equation}\label{eq:coldet}
\coldet(a_{ij})=\sum_{\s\in\mc{S}_r}\sgn(\s)\cdot a_{\s(1)1}\cdot a_{\s(2)2}\cdots a_{\s(n)n}. 
\end{equation}

We consider the Lie algebra $\mf{gl}_r$ and choose a basis $\{\mc{E}_{ij}:i,j\in[r]\}$ for it, where $\mc{E}_{ij}$ is the matrix whose only non-zero entry is in row $i$, column $j$, and it is equal to one. We think of $\mc{E}_{ij}$ as the inputs of an $r\times r$ matrix $\mc{E}$ with entries in $\mc{U}(\mf{gl}_r)$. We consider an auxiliary variable $z$, consider the diagonal matrix $\Delta=\diag(r-1-z,r-2-z,\cdots,1-z,-z)$ and define the polynomial $\mc{C}(z)\in\mc{U}(\mf{gl}_r)[z]$ using notation (\ref{eq:coldet}):
\begin{equation}\label{eq:defC(z)}
 \mc{C}(z)=\coldet(\mc{E}+\Delta).
\end{equation}
For $a\geq 0$ we write
\begin{equation}\label{eq:defza}
 [z]_a=z(z-1)\cdots(z-a+1)
\end{equation}
and define elements $\mc{C}_a\in\mc{U}(\mf{gl}_r)$, $a=0,\cdots,r$, by expanding the polynomial $\mc{C}(z)$ into a linear combination
\begin{equation}\label{eq:defCi}
 \mc{C}(z)=\sum_{a=0}^r (-1)^{r-a}\mc{C}_a\cdot[z]_{r-a}.
\end{equation}
In the case when $r=2$ we obtain
\[\mc{C}(z)=\coldet\begin{bmatrix}
 \mc{E}_{11}+1-z & \mc{E}_{12} \\
 \mc{E}_{21} & \mc{E}_{22}-z
\end{bmatrix} =
[z]_2-(\mc{E}_{11}+\mc{E}_{22})\cdot[z]+((\mc{E}_{11}+1)\cdot\mc{E}_{22}-\mc{E}_{21}\cdot\mc{E}_{12}),\textrm{ thus}
\]
\[\mc{C}_0=1,\quad\mc{C}_1=\mc{E}_{11}+\mc{E}_{22},\quad\mc{C}_2=(\mc{E}_{11}+1)\cdot\mc{E}_{22}-\mc{E}_{21}\cdot\mc{E}_{12}.\]
The elements $\mc{C}_a$, $a=1,\cdots,r$ are called \defi{the Capelli elements} of $\mc{U}(\mf{gl}_r)$, and $\mc{ZU}(\mf{gl}_r)$ is a polynomial algebra with generators $\mc{C}_1,\cdots,\mc{C}_r$. For $\ll\in\bb{Z}^r_{\dom}$, let $V_{\ll}$ denote an irreducible $\mf{gl}_r$-representation of highest weight $\ll$, and pick $v_{\ll}\in V_{\ll}$ to be a highest weight vector in $V_{\ll}$, so that
\begin{equation}\label{eq:hwvll}
\mc{E}_{ii}\cdot v_{\ll}=\ll_i\cdot v_{\ll},\quad \mc{E}_{ij}\cdot v_{\ll}=0\rm{ for }i<j. 
\end{equation}
Since $\mc{C}_a$ are central, their action on $V_{\ll}$ is by scalar multiplication, and the scalar (called the \defi{eigenvalue of $\mc{C}_a$ on $V_{\ll}$}) can be determined by just acting on $v_{\ll}$. To record this action more compactly, we will consider how $\mc{C}(z)$ acts on $v_{\ll}$. Expanding $\mc{C}(z)$ via (\ref{eq:coldet}), it follows from (\ref{eq:hwvll}) that the only term that doesn't annihilate $v_{\ll}$ is the product of diagonal entries in the matrix $\mc{E}+\Delta$, hence
\begin{equation}\label{eq:actionC(z)onVlam}
\mc{C}(z)\rm{ acts on }V_{\ll}[z]\rm{ by multiplication by }\prod_{i=1}^r(\ll_i+r-i-z). 
\end{equation}

We can think of $\mc{U}(\mf{gl}_r)$ in terms of generators and relations as follows: it is generated as a $\bb{C}$-algebra by $\mc{E}_{ij}$, $i,j\in[r]$, subject to the relations
\begin{equation}\label{eq:relsUglr}
[\mc{E}_{ij},\mc{E}_{kl}]=\d_{jk}\cdot\mc{E}_{il}-\d_{il}\cdot\mc{E}_{kj}, 
\end{equation}
where $[a,b]=ab-ba$ denotes the usual commutator, and $\d$ is the Kronecker delta function. For every complex number $u\in\bb{C}$, the substitutions
\[\mc{E}_{ij}\lra -\mc{E}_{ji}\rm{ for }i\neq j,\quad\mc{E}_{ii}\lra -\mc{E}_{ii}-u\]
preserve (\ref{eq:relsUglr}), so they define an involution $\mc{F}_u:\mc{U}(\mf{gl}_r)\lra\mc{U}(\mf{gl}_r)$ which we call \defi{the Fourier transform with parameter $u$}. We can apply $\mc{F}_u$ to $\mc{C}(z)$ and obtain
\begin{equation}\label{eq:FourierC(z)}
\mc{F}_u\mc{C}(z)=\coldet(-\mc{E}^t-u\cdot\Id_r+\Delta), 
\end{equation}
where $\mc{E}^t$ is the transpose of $\mc{E}$, and $\Id_r$ denotes the $r\times r$ identity matrix. The Fourier transforms $\mc{F}_u\mc{C}_1,\cdots,\mc{F}_u\mc{C}_r$ of the Capelli elements form another set of polynomial generators for $\mc{ZU}(\mf{gl}_r)$, hence they act by scalar multiplication on any irreducible $\mf{gl}_r$-representation $V_{\ll}$. To determine the scalars, we will consider the action on a lowest weight vector $w_{\ll}\in V_{\ll}$, so that
\begin{equation}\label{eq:lwvll}
\mc{E}_{ii}\cdot w_{\ll}=\ll_{r+1-i}\cdot w_{\ll},\quad \mc{E}_{ji}\cdot w_{\ll}=0\rm{ for }i<j.
\end{equation}
Expanding (\ref{eq:FourierC(z)}) via (\ref{eq:coldet}), it follows from (\ref{eq:lwvll}) that the action of $\mc{F}_u\mc{C}_a$ on $V_{\ll}$ is encoded by the fact that
\begin{equation}\label{eq:actionFC(z)onVlam}
\mc{F}_u\mc{C}(z)\rm{ acts on }V_{\ll}[z]\rm{ by multiplication by }\prod_{i=1}^r(-\ll_{r+1-i}-u+r-i-z).
\end{equation}

\begin{lemma}\label{lem:Fs^r}
 For $s\in\bb{Z}$, let $\ll=(s^r)$ denote the dominant weight with all $\ll_i=s$, and for $a=1,\cdots,r$ let $P_a(s)$ (resp. $\mc{F}_uP_a(s)$) denote the eigenvalue of $\mc{C}_a$ (resp. $\mc{F}_u\mc{C}_a$) on $V_{\ll}$. We have that $P_a(s)$ and $\mc{F}_uP_a(s)$ are polynomial functions in $s$, and as such $\mc{F}_uP_a(s)=P_a(-s-u)$.
\end{lemma}

\begin{proof}
 If we let $P(s,z)=\sum_{a=0}^r (-1)^{r-a}P_a(s)\cdot[z]_{r-a}$ then it follows from (\ref{eq:defCi}) and (\ref{eq:actionC(z)onVlam}) that
 \[P(s,z)=\prod_{i=1}^r(s+r-i-z).\]
 Expanding the right hand side as a linear combination of $[z]_0,[z]_1,\cdots,[z]_r$ shows that $P_a(s)$ is a polynomial in $s$. We define $\mc{F}_uP(s,z)$ by replacing $P_a(s)$ with $\mc{F}_uP_a(s)$ and obtain using (\ref{eq:actionFC(z)onVlam}) that
 \[\mc{F}_uP(s,z)=\prod_{i=1}^r(-s-u+r-i-z).\]
 Since $\mc{F}_uP(s,z)=P(-s-u,z)$, the conclusion follows.
\end{proof}

\begin{lemma}\label{lem:Fs^r-1}
 For $s\in\bb{Z}_{\geq 0}$, let $\ll=(s^{r-1})$ denote the partition with $\ll_1=\cdots=\ll_{r-1}=s$, $\ll_r=0$, and for $a=1,\cdots,r$ let $Q_a(s)$ (resp. $\mc{F}_{r-1}Q_a(s)$) denote the eigenvalue of $\mc{C}_a$ (resp. $\mc{F}_{r-1}\mc{C}_a$) on $V_{\ll}$. We have that $Q_a(s)$ and $\mc{F}_{r-1}Q_a(s)$ are polynomial functions in $s$, and as such $\mc{F}_{r-1}Q_a(s)=Q_a(-s-r)$.
\end{lemma}

\begin{proof}
 We define $Q(s,z)$ and $\mc{F}_{r-1}Q(s,z)$ as in the proof of Lemma~\ref{lem:Fs^r} and obtain using (\ref{eq:actionC(z)onVlam}), (\ref{eq:actionFC(z)onVlam}) that
\[Q(s,z)=\left(\prod_{i=1}^{r-1}(s+r-i-z)\right)\cdot(0-z)=(s+r-1-z)\cdot(s+r-2-z)\cdots(s+1-z)\cdot(-z),\]
\[
\begin{aligned}
\mc{F}_{r-1}Q(s,z)&=(-0-(r-1)+r-1-z)\cdot\left(\prod_{i=2}^r(-s-(r-1)+r-i-z)\right) \\
&=(-z)\cdot(-s-1-z)\cdot(-s-2-z)\cdots(-s-r+1-z).
\end{aligned}
\]
It is immediate to check that $\mc{F}_{r-1}Q(s,z)=Q(-s-r,z)$, from which the conclusion follows.
\end{proof}

\subsection{A little linear algebra}\label{subsec:linalg}

We let $X_{m,n}$, $m\geq n$, denote the vector space of $m\times n$ matrices, write $Z_{m,n}$ for the subvariety of $X_{m,n}$ consisting of matrices of rank at most $n-1$, and let $U\subset X_{m,n}$ denote the open affine subset consisting of matrices $u=(u_{ij})$ with $u_{11}\neq 0$.

\begin{lemma}\label{lem:localizemxnmat}
 There exists an isomorphism of algebraic varieties
\[\pi:U\cap Z_{m,n}\lra \bb{C}^*\times\bb{C}^{m-1}\times\bb{C}^{n-1}\times Z_{m-1,n-1}.\]
\end{lemma}

\begin{proof}
 We define $\pi:U\to\bb{C}^*\times\bb{C}^{m-1}\times\bb{C}^{n-1}\times X_{m-1,n-1}$ via $\pi(u)=(t,\vec{c},\vec{r},M)$ where if $u=(u_{ij})$ then
 \[t=u_{11},\ \vec{c}=(u_{21},u_{31},\cdots,u_{m1}),\ \vec{r}=(u_{12},u_{13},\cdots,u_{1n}),\]
 \[\ M_{ij}=\det_{\{1,i+1\},\{1,j+1\}}\rm{ for }i\in[m-1],j\in[n-1],\]
 where $\det_{\{1,i+1\},\{1,j+1\}}=u_{11}\cdot u_{i+1,j+1}-u_{1,j+1}\cdot u_{i+1,1}$ is the determinant of the $2\times 2$ submatrix of $u$ obtained by selecting rows $1,i+1$ and columns $1,j+1$. It follows for instance from \cite[Section~3.4]{johnson} that the map $\pi$ is an isomorphism, and that it sends $U\cap Z_{m,n}$ onto $\bb{C}^*\times\bb{C}^{m-1}\times\bb{C}^{n-1}\times Z_{m-1,n-1}$, which yields the desired conclusion.
\end{proof}

We let $X_n$ denote the vector space of $(2n+1)\times(2n+1)$ skew-symmetric matrices, and define $Z_n\subset X_n$ to be the subvariety of matrices of rank at most $(2n-2)$. We let $U\subset X_n$ denote the open affine subset defined by matrices $(u_{ij})$ with $u_{12}\neq 0$.

\begin{lemma}\label{lem:localizeskewsym}
 There exists an isomorphism of algebraic varieties
\[\pi:U\cap Z_{n}\lra \bb{C}^*\times\bb{C}^{2n-1}\times\bb{C}^{2n-1}\times Z_{n-1}.\]
\end{lemma}

\begin{proof}
 We define $\pi:U\to\bb{C}^*\times\bb{C}^{2n-1}\times\bb{C}^{2n-1}\times X_{n-1}$ via $\pi(u)=(t,\vec{c},\vec{r},M)$ where if $u=(u_{ij})$ then
 \[t=u_{12},\ \vec{c}=(u_{13},u_{14},\cdots,u_{1,2n+1}),\ \vec{r}=(u_{23},u_{24},\cdots,u_{2,2n+1}),\]
 \[\ M_{ij}=\frac{\Pf_{\{1,2,i+2,j+2\}}}{u_{12}}\rm{ for }1\leq i,j\leq 2n-1,\]
 where $\Pf_{\{1,2,i+2,j+2\}}$ is the Pfaffian of the $4\times 4$ principal skew-symmetric submatrix of $u$ obtained by selecting the rows and columns of $u$ indexed by $1,2,i+2$ and $j+2$. Since 
 \[M_{ij}=u_{i+2,j+2}-(u_{1,i+2}\cdot u_{2,j+2}-u_{1,j+2}\cdot u_{2,i+2})/u_{12}\]
 one can solve for $u_{i+2,j+2}$ in terms of the entries of $M,\vec{r},\vec{c}$ and $u_{12}$ in order to define the inverse of $\pi$, which is therefore an isomorphism. We consider the $(2n+1)\times(2n+1)$ matrix
 \[C=\left[
\begin{array}{cc|cccc}
0 & 1 & u_{23}/u_{12} & u_{24}/u_{12} & \cdots & u_{2,2n+1}/u_{12} \\
1/u_{12} & 0 & -u_{13}/u_{12} & -u_{14}/u_{12} & \cdots & -u_{1,2n+1}/u_{12} \\ \hline
0 & 0 & 1 & 0 & \cdots & 0\\
0 & 0 & 0 & 1 & \cdots & 0\\
\vdots & \vdots & \vdots & \vdots & \ddots & \vdots \\
0 & 0 & 0 & 0 & \cdots & 1\\
\end{array}
\right]  
\]
Writing $\vec{0}$ for zero row/column vectors of size $(2n-1)$, we have (see also \cite[Lemma~1.1]{joz-pra})
\[C^t\cdot u\cdot C = \left[\begin{array}{cc|c}
0 & -1 & \vec{0} \\
1 & 0 & \vec{0} \\ \hline
\vec{0} & \vec{0} & M \\
\end{array}
\right]  
\]
Since $\rank(u)=\rank(C^t\cdot u\cdot C)=\rank(M)+2$, it follows that $\pi$ sends $U\cap Z_{n}$ onto $\bb{C}^*\times\bb{C}^{2n-1}\times\bb{C}^{2n-1}\times Z_{n-1}$, so it restricts to the desired isomorphism.
\end{proof}

\subsection{The $b$-function of an affine scheme}\label{subsec:bsatoideal}

In this section we review the results and definitions from \cite{budur-mustata-saito} that are most relevant for our calculations. Let $X=\bb{A}^N$ be the $N$-dimensional affine space, and write $S=\bb{C}[x_1,\cdots,x_N]$ for the coordinate ring of $X$, and $\D_X=\bb{C}[x_1,\cdots,x_N,\p_1,\cdots,\p_N]$ ($\p_i=\frac{\p}{\p x_i}$) for the corresponding Weyl algebra of differential operators. For a collection $f=(f_1,\cdots,f_r)$ of non-zero polynomials in $S$, we consider a set of independent commuting variables $s_1,\cdots,s_r$, one for each $f_i$. We form the $\D_X[s_1,\cdots,s_r]$-module
\begin{equation}\label{eq:defBs}
B_{\ul s}^f=S_{f_1\cdots f_r}[s_1,\cdots,s_r]\cdot f^{\ul s}, 
\end{equation}
where $S_{f_1\cdots f_r}$ denotes the localization of $S$ at the product of the $f_i$'s, and $f^{\ul s}$ stands for the formal product $f_1^{s_1}\cdots f_r^{s_r}$. $B_{\ul s}^f$ is a free rank one $S_{f_1\cdots f_r}[s_1,\cdots,s_r]$-module with generator $f^{\ul s}$, which admits a natural action of $\D_X$: the partial derivatives $\p_i$ act on the generator $f^{\ul s}$ via
\[\p_i\cdot f^{\ul s} = \sum_{j=1}^r \frac{s_j\cdot(\p_i\bul f_j)}{f_j}\cdot f^{\ul s}.\]
Writing $s=s_1+\cdots+s_r$, the \defi{Bernstein--Sato polynomial} (or \defi{$b$-function}) $b_f(s)$ is the monic polynomial of the lowest degree in $s$ for which $b_f(s)\cdot f^{\ul s}$ belongs to the $\D_X[s_1,\cdots,s_r]$-submodule of $B_{\ul s}^f$ generated by all expressions
\[\hat{c}_{\ul s}\cdot\prod_{i=1}^r f_i^{s_i+c_i},\]
where $\hat{c}=(c_1,\cdots,c_r)$ runs over the $r$-tuples in $\bb{Z}^r$ with $c_1+\cdots+c_r=1$ (for short $|\hat{c}|=1$), and
\begin{equation}\label{eq:defchat}
\hat{c}_{\ul s}=\prod_{c_i<0} s_i\cdot(s_i-1)\cdots(s_i+c_i+1). 
\end{equation}
 Equivalently, $b_f(s)$ is the monic polynomial of lowest degree for which there exist a finite set of tuples $\hat{c}\in\bb{Z}^r$ with $|\hat{c}|=1$, and corresponding operators $P_{\hat{c}}\in\D_X[s_1,\cdots,s_r]$ such that
\begin{equation}\label{eq:genbfseqn}
 \sum_{\hat{c}} P_{\hat{c}}\cdot\hat{c}_{\ul s}\cdot\prod_{i=1}^r f_i^{s_i+c_i}=b_f(s)\cdot f^{\ul s}.
\end{equation}
 
Just as in the case $r=1$ (of a single hypersurface), $b_f(s)$ exists and is a polynomial whose roots are negative rational numbers. Moreover, $b_f(s)$ only depends on the ideal $I$ generated by $f_1,\cdots,f_r$, which is why we'll often write $b_I(s)$ instead of $b_f(s)$. Furthermore, if we let $Z\subset X$ denote the subscheme defined by $f_1,\cdots,f_r$, and if we define
\begin{equation}\label{eq:defbZ}
b_Z(s)=b_I(s-\rm{codim}_X(Z)) 
\end{equation}
then $b_Z(s)$ only depends on the affine scheme $Z$ and not on its embedding in an affine space. The polynomial $b_Z(s)$ is called the \defi{Bernstein--Sato polynomial of $Z$} (or the \defi{$b$-function of $Z$}), and is meant as a measure of the singularities of $Z$: the higher the degree of $b_Z(s)$, the worse are the singularities of $Z$. For instance, one has that $T$ is smooth if and only if $b_T(s)=s$. Moreover, it follows from \cite[Theorem~5]{budur-mustata-saito} that for any $Z$ and any smooth $T$ we have
\begin{equation}\label{eq:bZxT=bZ}
b_{Z\times T}(s)=b_Z(s). 
\end{equation}
It will be important to note also that if $Z$ is irreducible and $Z=Z_1\cup\cdots\cup Z_k$ is an open cover of $Z$ then
\begin{equation}\label{eq:bZ=lcmbZi}
b_Z(s)=\rm{lcm}\{b_{Z_i}(s):i=1,\cdots,k\}. 
\end{equation}
A modification of the above formula is shown in \cite{budur-mustata-saito} to hold even when $Z$ is reducible, and in fact can be used to define a $b$-function for not necessarily affine or irreducible schemes $Z$: this generality is not relevant for this article so we won't discuss it further. Combining (\ref{eq:bZxT=bZ}) and (\ref{eq:bZ=lcmbZi}) with the results and notation from Section~\ref{subsec:linalg}, we conclude that
\begin{equation}\label{eq:divinduction}
 b_{Z_{m-1,n-1}}(s) \rm{ divides } b_{Z_{m,n}}(s),\rm{ and }b_{Z_{n-1}}(s) \rm{ divides } b_{Z_n}(s).
\end{equation}

\section{Bounding the $b$-function}\label{sec:boundingbfun}

In this section we discuss some methods for bounding the $b$-function from above and below. As a consequence we obtain formulas for the $b$-function of the ideal of maximal minors of the generic $(n+1)\times n$ matrix, and for the $b$-function of the ideal of sub-maximal Pfaffians of a generic skew-symmetric matrix of odd size.

\subsection{Lower bounds}\label{subsec:lowerbounds}

In order to obtain lower bounds for a $b$-function, it is important to be able to identify certain factors of the $b$-function which are easier to compute. One instance of this is given in equation (\ref{eq:bZ=lcmbZi}): the $b$-function of $Z$ is divisible by the $b$-function of any affine open subscheme. In this section we note that sometimes it is possible to identify roots of the $b$-function (i.e. linear factors) by showing an appropriate inclusion of $\D$-modules. As before $f=(f_1,\cdots,f_r)\in S^r$, and $I\subset S$ is the ideal generated by the $f_i$'s.

For $\a\in\bb{Z}$ we define $F_{\a}$ to be the $\D_X$-submodule of $S_{f_1\cdots f_r}$ generated by
\[f^{\ul\a}=\prod_{i=1}^r f_i^{\a_i},\rm{ where }\ul{\a}=(\a_1,\cdots,\a_r)\in\bb{Z}^r,\ \a_1+\cdots+\a_r=\a.\]
It is clear that $F_{\a+1}\subseteq F_{\a}$ for every $\a\in\bb{Z}$. We have moreover:

\begin{proposition}\label{prop:rootbfun}
 If $\a\in\bb{Z}$ and if there is a strict inclusion $F_{\a+1}\subsetneq F_{\a}$ then $\a$ is a root of $b_f(s)$.
\end{proposition}

\begin{proof}
 By the definition of $b_f(s)$, there exist tuples $\hat{c}$ and operators $P_{\hat{c}}\in\D_X[s_1,\cdots,s_r]$ such that (\ref{eq:genbfseqn}) holds. Assume now that $F_{\a+1}\subsetneq F_{\a}$ for some $\a\in\bb{Z}$, and consider any integers $\a_1,\cdots,\a_r$ with $\a_1+\cdots+\a_r=\a$. There is a natural $\D_X$-module homomorphism 
 \begin{equation}\label{eq:specpi}
\pi:S_{f_1\cdots f_r}[s_1,\cdots,s_r]\cdot f^{\ul s}\longrightarrow S_{f_1\cdots f_r},\rm{ defined by }\pi(s_i)=\a_i.  
 \end{equation}
 Applying $\pi$ to (\ref{eq:genbfseqn}) we find that $b_f(\a)\cdot f^{\ul\a}\in F_{\a+1}$. If $b_f(\a)\neq 0$ then we can divide by $b_f(\a)$ and obtain that $f^{\ul\a}\in F_{\a+1}$ for all $\ul{\a}$ with $|\ul{\a}|=\a$. Since the elements $f^{\ul\a}$ generate $F_{\a}$ it follows that $F_{\a}\subseteq F_{\a+1}$ which is a contradiction. We conclude that $b_f(\a)=0$, i.e. that $\a$ is a root of $b_f(s)$.
\end{proof}

We write $H_I^{\bullet}(S)$ for the local cohomology groups of $S$ with support in the ideal $I$. Proposition~\ref{prop:rootbfun} combined with non-vanishing results for local cohomology can sometimes be used to determine roots of the $b$-function as follows:

\begin{corollary}\label{cor:rootfromloccoh}
 If $b_I(s)$ has no integral root $\a$ with $\a<-r$, and if $H_I^r(S)\neq 0$ then $b_I(-r)=0$.
\end{corollary}

\begin{proof} For every $\a\in\bb{Z}$, $\a<-r$, and every $\ul{\a}=(\a_1,\cdots,\a_r)$ with $\a=\a_1+\cdots+\a_r$, we can apply the specialization map (\ref{eq:specpi}) to the equation (\ref{eq:genbfseqn}) to conclude that $b_I(\a)\cdot f^{\ul\a}\in F_{\a+1}$. Since $b_I(\a)\neq 0$ by assumption, we conclude that $f^{\ul\a}\in F_{\a+1}$ for all such $\a$, and therefore $F_{\a}=F_{\a+1}$. It follows that
\[F_{-r}=F_{-r-1}=F_{-r-2}=\cdots=S_{f_1\cdots f_r},\]
since the localization $S_{f_1\cdots f_r}$ is the union of all $F_{\a}$, $\a\leq-r$.
 
By Proposition~\ref{prop:rootbfun}, in order to show that $b_I(-r)=0$, it is enough to show that $F_{-r+1}\subsetneq F_{-r}$, which by the above is equivalent to proving that $F_{-r+1}$ does not coincide with the localization $S_{f_1\cdots f_r}$. Consider any generator $f^{\ul\a}$ of $F_{-r+1}$, corresponding to a tuple $\ul{\a}\in\bb{Z}^r$ with $\a_1+\cdots+\a_r=-r+1$. At least one of the $\a_i$'s has to be nonnegative, so that $f^{\ul\a}$ belongs to $S_{f_1\cdots\hat{f_i}\cdots f_r}$, the localization of $S$ at a product of all but one of the generators $f_i$. This shows that
\begin{equation}\label{eq:incF-r+1S_f}
F_{-r+1}\subseteq\sum_{i=1}^r S_{f_1\cdots\hat{f_i}\cdots f_r}.
\end{equation}
Using the \v Cech complex description of local cohomology, and the assumption that $H_I^r(S)\neq 0$, we conclude that there is a strict inclusion
\[\sum_{i=1}^r S_{f_1\cdots\hat{f_i}\cdots f_r}\subsetneq S_{f_1\cdots f_r}.\]
Combining this with (\ref{eq:incF-r+1S_f}) we conclude that $F_{-r+1}\subsetneq F_{-r}=S_{f_1\cdots f_r}$, as desired.
\end{proof}

\subsection{Upper bounds}\label{subsec:upperbounds}

Obtaining upper bounds for $b$-functions is in general a difficult problem, since most of the time it involves determining the operators $P_{\hat{c}}$ in (\ref{eq:genbfseqn}). In the presence of a large group of symmetries, invariant differential operators are natural candidates for such operators, and the problem becomes more tractable. As in Section~\ref{subsec:invDmods}, $G$ is a connected reductive linear algebraic group, and $\mf{g}$ is its Lie algebra.

\begin{definition}\label{def:irreduciblef}
 A tuple $f=(f_1,\cdots,f_r)\in S^r$ is said to be \defi{multiplicity-free (for the $G$-action)} if
\begin{itemize}
\item[(a)] For every nonnegative integer $\a$, the polynomials
\[f^{\ul\a}=f_1^{\a_1}\cdots f_r^{\a_r},\rm{ for }\ul{\a}=(\a_1,\cdots,\a_r)\in\bb{Z}^r_{\geq 0}\rm{ satisfying }\a_1+\cdots+\a_r=\a,\]
span an irreducible $G$-subrepresentation $V_{\a}\subset S$.
\item[(b)] For every $\a\in\bb{Z}_{\geq 0}$, the multiplicity of $V_{\a}$ inside $S$ is equal to one.
\end{itemize}
\end{definition}

A typical example of a multiplicity-free tuple arises in the case $r=1$ from a semi-invariant on a prehomogeneous vector space. In this case the computations for the Bernstein-Sato polynomials have been pursued thoroughly (see for example \cites{kimu,prehomogeneous}). Our definition gives a natural generalization to tuples with $r>1$ entries. We have the following:

\begin{proposition}\label{prop:DfgivesPfs}
Consider a multiplicity-free tuple $f=(f_1,\cdots,f_r)$ for some $G$, and a $G$-invariant differential operator $D_f=\sum_{i=1}^r g_i\cdot f_i$, where $g_i\in\D_X$. If we let $s=s_1+\cdots+s_r$ then there exists a polynomial $P_f(s)\in\bb{C}[s]$ such that
\[D_f\cdot f^{\ul s} = P_f(s)\cdot f^{\ul s},\]
and moreover we have that $b_{f}(s)$ divides $P_f(s)$.
\end{proposition}

\begin{proof}
Since the action of $D_f$ preserves $B_{\ul s}^f$, there exists an element $Q\in S_{f_1\cdots f_r}[s_1,\cdots,s_r]$ with the property
\[D_f\cdot f^{\ul s}=Q\cdot f^{\ul s}.\]
The goal is to show that, as a polynomial in $s_1,\cdots,s_r$, $Q=Q(s_1,\cdots,s_r)$ has coefficients in $\bb{C}$, and moreover that it can be expressed as a polynomial only in $s=s_1+\cdots+s_r$. For this, it suffices to check that:

(a) $Q(\a_1,\cdots,\a_r)\in\bb{C}$ for every $\a_1,\cdots,\a_r\in\bb{Z}_{\geq 0}$.

(b) For $\a_i$ as in (a), $Q(\a_1,\cdots,\a_r)$ only depends on $\a=\a_1+\cdots+\a_r$.

Let $\a_1,\cdots,\a_r$ be arbitrary non-negative integers, and write $\a=\a_1+\cdots+\a_r$. Since $V_{\a}$ is irreducible, $\scpr{V_{\a}}{S}=1$, and $D_f$ is $G$-invariant, it follows from Schur's Lemma that $D_f$ acts on $V_{\a}$ by multiplication by a scalar, i.e. $Q(\a_1,\cdots,\a_r)\in\bb{C}$ is a scalar that only depends on $\a$, so conditions (a) and (b) are satisfied.

To see that $b_f(s)$ divides $P_f(s)$, it suffices to note that $D_f\cdot f^{\ul s} = P_f(s)\cdot f^{\ul s}$ can be rewritten in the form (\ref{eq:genbfseqn}), where the sum is over tuples $\hat{c}=(0,\cdots,0,1,0,\cdots,0)$ with $c_i=1$, $c_j=0$ for $j\neq i$, with corresponding operator $P_{\hat{c}}=g_i$. Since $b_f(s)$ is the lowest degree polynomial for which (\ref{eq:genbfseqn}) holds, it follows that $b_f(s)$ divides $P_f(s)$.
\end{proof}

\subsection{Maximal minors}

In this section $X=X_{m,n}$ is the vector space of $m\times n$ matrices, $m\geq n$. The group $G=\GL_m\times\GL_n$ acts on $X$ via row and column operations. The coordinate ring of $X$ is $S=\bb{C}[x_{ij}]$, and we consider the tuple $d=(d_K)_{K\in{[m]\choose n}}$ of maximal minors defined in (\ref{eq:defdK}). The tuple $d$ is multiplicity-free for the $G$-action, where for $\a\in\bb{Z}_{\geq 0}$, the corresponding representation $V_{\a}$ in Definition~\ref{def:irreduciblef} is $S_{(\a^n)}\bb{C}^m\oo S_{(\a^n)}\bb{C}^n$ from (\ref{eq:Cauchy}) (see for instance \cite[Thm.~6.1]{deconcini-eisenbud-procesi}). We associate to $d$ the invariant differential operator $D_d$ in (\ref{eq:defDd}) and by Proposition~\ref{prop:DfgivesPfs} there exists a polynomial $P_d(s)$ with
\begin{equation}\label{eq:DdgivesPd}
D_d\cdot d^{\ul s}=P_d(s)\cdot d^{\ul s}.
\end{equation}

\begin{theorem}\label{thm:Pds}
 With the notation above, we have that
\begin{equation}\label{eq:Pds}
P_d(s)=\prod_{i=m-n+1}^m (s+i). 
\end{equation}
\end{theorem}

\begin{proof} In order to compute $P_d(s)$, it suffices to understand the action of $D_d$ on $d_L^s$ for some fixed $L\in{[m]\choose n}$ (this corresponds to letting $s_K=0$ for $K\neq L$ in (\ref{eq:DdgivesPd})). We consider instead the action of the operator $D_{\p}$ in (\ref{eq:defDd}), and note that by Cayley's identity \cite[(1.1)]{css} one has
\[\p_K\cdot d_L^s=0\rm{ for }K\neq L,\quad \p_L\cdot d_L^s=\left(\prod_{i=0}^{n-1}(s+i)\right)\cdot d_L^{s-1},\]
which implies
\begin{equation}\label{eq:Dd*dL^s}
D_{\p}\cdot d_L^s = \left(\prod_{i=0}^{n-1}(s+i)\right)\cdot d_L^s. 
\end{equation}
Let $\mc{F}:\D_X\lra\D_X$ denote the (usual) Fourier transform, defined by $\mc{F}(x_{ij})=\p_{ij}$, $F(\p_{ij})=-x_{ij}$, and note that $D_d=(-1)^n\cdot\mc{F}(D_{\p})$. We will obtain $P_d(s)$ by applying the Fourier transform to (\ref{eq:Dd*dL^s}).

For $i,j\in [n]$, we consider the polarization operators
\[E_{ij}=\ds_{k=1}^m x_{ki}\cdot\partial_{kj}.\]
The action of the Lie algebra $\mf{gl}_n\subset\mf{gl}_m\oplus\mf{gl}_n$ on $X$ induces a map $\tau:\mc{U}(\mf{gl}_n)\to\D_X$ as in (\ref{eq:deftau}), sending $\tau(\mc{E}_{ij})=E_{ij}$ for all $i,j$. The Fourier transform sends
\[\mc{F}(E_{ij})=-E_{ji}\rm{ for }i\neq j,\quad \mc{F}(E_{ii})=-E_{ii}-m,\]
so using the notation in Section~\ref{subsec:CapelliFourier} we obtain a commutative diagram
\[
 \xymatrix{
 \mc{U}(\mf{gl}_n) \ar^{\mc{F}_m}[r] \ar_{\tau}[d] & \mc{U}(\mf{gl}_n) \ar^{\tau}[d]\\
 \D_X \ar^{\mc{F}}[r] & \D_X\\
 }
\]
Since $D_{\p}$ is in $\tau(\mc{ZU}(\mf{gl}_n))$ (it is in fact equal to $\tau(\mc{C}_n)$ by \cite[(11.1.9)]{howe-umeda}), it follows from (\ref{eq:Dd*dL^s}), from the commutativity of the above diagram and from Lemma~\ref{lem:Fs^r} with $r=n$ and $u=m$ that
\[D_d\cdot d_K^s = \left((-1)^n\prod_{i=0}^{n-1}(-s-m+i)\right)\cdot d_K^s=\left(\prod_{i=m-n+1}^m (s+i)\right)\cdot d_K^s,\]
which concludes the proof of our theorem.
\end{proof}

\begin{remark}
 A more direct way to prove (\ref{eq:Pds}) is to use for instance \cite[Prop.~1.2]{css-capelli} in order to obtain a determinantal representation for the operator $D_d$, namely
\[D_d=\coldet
\begin{pmatrix}
E_{11}+m & E_{12} & \cdots & E_{1n}\\
E_{21} & E_{22}+m-1 & \cdots & E_{2n}\\
\vdots & \vdots & \ddots & \vdots\\
E_{n1} & E_{n2} & \cdots & E_{nn}+m-n+1
\end{pmatrix},\] 
from which the conclusion follows easily. The advantage of our proof of Theorem~\ref{thm:Pds} is that it applies equally to the case of sub-maximal Pfaffians in Section~\ref{subsec:Pfaffians}, where we are not aware of a more direct approach.
\end{remark}

\subsection*{Almost square matrices}\label{sec:almostsquare}

In the case of $(n+1)\times n$ matrices, we can show that the lower and upper bounds obtained by the techniques described above agree, and we obtain the following special instance of the Theorem on Maximal Minors described in the Introduction:

\begin{theorem}\label{thm:n+1byn}
 If $d$ is the tuple of maximal minors of the generic $(n+1)\times n$ matrix then its $b$-function is
\[b_d(s)=\prod_{i=2}^{n+1} (s+i).\]
\end{theorem}

\begin{proof}
 We have by Proposition~\ref{prop:DfgivesPfs} and Theorem~\ref{thm:Pds} that $b_d(s)$ divides the product $(s+2)\cdots (s+n+1)$. If we write $Z_{n+1,n}$ for the variety of $(n+1)\times n$ matrices of rank smaller than $n$ as in Section~\ref{subsec:linalg} then the defining ideal of $Z_{n+1,n}$ is generated by the entries of $d$. Since $Z_{n+1,n}$ has codimension two inside $X_{n+1,n}$, $b_{Z_{n+1,n}}(s)=b_d(s-2)$ by (\ref{eq:defbZ}), and thus it suffices to show that
 \begin{equation}\label{eq:lowerbdbZn}
  b_{Z_{n+1,n}}(s)\rm{ is divisible by }\prod_{i=0}^{n-1}(s+i).  
 \end{equation}
 By induction on $n$, we may assume that $b_{Z_{n,n-1}}=\prod_{i=0}^{n-2}(s+i)$. Taking into account (\ref{eq:divinduction}) we are left with proving that $(-n+1)$ is a root of $b_{Z_{n+1,n}}(s)$, or equivalently that $(-n-1)$ is a root of $b_d(s)$. To do this we apply Corollary~\ref{cor:rootfromloccoh} with $r=n+1$, and $I$ the defining ideal of $Z_{n+1,n}$. It follows from \cite[Thm.~5.10]{witt} or \cite[Thm.~4.5]{raicu-weyman-witt} that $H^{n+1}_I(S)\neq 0$, so the Corollary applies and concludes our proof.
\end{proof}

\begin{remark}\label{rem:quiv}
An alternative approach to proving Theorem~\ref{thm:n+1byn} goes by first computing the $b$-function of several variables associated to $d_1,\cdots,d_{n+1}$ (see \cite[Lemma 1.9]{bub1}). The space $X_{n+1,n}$ is prehomogeneous under the action of the smaller group $(\complex^*)^{n+1} \times \GL_n(\complex)$. We will use freely some notions from \cite{bub1}. The maximal minors $d_1,\cdots,d_{n+1}$ can be viewed as semi-invariants for the following quiver with $n+2$ vertices and dimension vector 
\[\xymatrix{
1 & 1 & \cdots & 1 & 1 \\
 & & \ar[ull]\ar[ul]\ar[ur]\ar[urr] n & &
}\]

The dimension vector is preinjective, hence by \cite[Proposition 5.4(b)]{bub1} we can compute the $b$-function of several variables using reflection functors \cite[Theorem 5.3]{bub1}:
$$b_{d}(\underline{s})=[s]_{n-1,n}^{1,1,\cdots,1} \cdot [s]_1^{1,0,\cdots,0} \cdot [s]_1^{0,1,\cdots,0}\cdots [s]_1^{0,0,\cdots,1}.$$
This means that we have formulas
$$d_i^* \cdot d_i \cdot d^{\ul s} = (s_i+1)(s+2)(s+3)\cdots (s+n)\cdot d^{\ul s},$$
which, together with Lemma~\ref{lem:equivdefaps} below gives readily the Bernstein-Sato polynomial of the ideal. Such relations between $b$-functions of several variables and Bernstein-Sato polynomials of ideals have been investigated in \cite{bub2}.
\end{remark}

\subsection{Sub-maximal Pfaffians}\label{subsec:Pfaffians}

In this section $X=X_{n}$ is the vector space of $(2n+1)\times(2n+1)$ skew-symmetric matrices, with the natural action of $G=\GL_{2n+1}$. The coordinate ring of $X$ is $S=\bb{C}[x_{ij}]$ with $1\leq i<j\leq 2n+1$. We consider the tuple $d=(d_1,d_2,\cdots,d_{2n+1})$, where $d_i$ is the Pfaffian of the skew-symmetric matrix obtained by removing the $i$-th row and column of the generic skew-symmetric matrix $(x_{ij})_{i,j\in[2n+1]}$ (with the convention $x_{ji}=-x_{ij}$ and $x_{ii}=0$). The tuple $d$ is multiplicity-free for the $G$-action, where for $\a\in\bb{Z}_{\geq 0}$, the corresponding representation $V_{\a}$ in Definition~\ref{def:irreduciblef} is $S_{(\a^{2n})}\bb{C}^{2n+1}$ from (\ref{eq:skewCauchy}) (see for instance \cite[Thm.~4.1]{abeasis-delfra}). We associate to $d$ the invariant differential operator
\[D_d=\sum_{i=1}^{2n+1}d_i^*\cdot d_i,\]
and by Proposition~\ref{prop:DfgivesPfs} there exists a polynomial $P_d(s)$ with
\begin{equation}\label{eq:skewDdgivesPd}
D_d\cdot d^{\ul s}=P_d(s)\cdot d^{\ul s}.
\end{equation}

\begin{theorem}\label{thm:bfunskew}
 If $d$ is the tuple of sub-maximal Pfaffians of the generic $(2n+1)\times(2n+1)$ skew-symmetric matrix, then
\begin{equation}\label{eq:bfunskew}
 b_d(s)=P_d(s)=\prod_{i=0}^{n-1} (s+2i+3). 
\end{equation}
\end{theorem}

\begin{proof}
 We begin by showing, using the strategy from the proof of Theorem~\ref{thm:Pds}, that $P_d(s)=\prod_{i=0}^{n-1} (s+2i+3)$. We have a commutative diagram
\[
 \xymatrix{
 \mc{U}(\mf{gl}_{2n+1}) \ar^{\mc{F}_{2n}}[r] \ar_{\tau}[d] & \mc{U}(\mf{gl}_{2n+1}) \ar^{\tau}[d]\\
 \D_X \ar^{\mc{F}}[r] & \D_X\\
 }
\]
If we let $D_{d^*}=\sum_{i=1}^{2n+1}d_i\cdot d_i^*$ then $D_d=(-1)^n\cdot\mc{F}(D_{d^*})$. It follows from \cite[Thm.~2.3]{css} that
\[d_i^*\cdot d_0^s=0\rm{ for }i\neq 0,\quad d_0^*\cdot d_0^s=\left(\prod_{i=0}^{n-1}(s+2i)\right)\cdot d_0^{s-1},\]
from which we obtain
\[D_{d^*}\cdot d_0^s = \left(\prod_{i=0}^{n-1}(s+2i)\right)\cdot d_0^s.\]
Since $D_{d^*}$ is in $\tau(\mc{ZU}(\mf{gl}_{2n+1}))$ by \cite[Cor.~11.3.19]{howe-umeda}, it follows from Lemma~\ref{lem:Fs^r-1} with $r=2n+1$ that
\[D_d\cdot d_0^s = \left((-1)^n\cdot\prod_{i=0}^{n-1}(-s-2n-1+2i)\right)\cdot d_0^s,\]
from which we obtain
\begin{equation}\label{eq:skewPds}
P_d(s)=\prod_{i=0}^{n-1} (s+2i+3).
\end{equation}
 
Using the notation in Section~\ref{subsec:linalg} we have that $b_d(s)=b_{Z_n}(s+3)$ since $Z_n$ has codimension three in~$X_n$, so (\ref{eq:bfunskew}) is equivalent to $b_{Z_n}(s)=\prod_{i=0}^{n-1}(s+2i)$. By induction on $n$ we have $b_{Z_{n-1}}(s)=\prod_{i=0}^{n-2}(s+2i)$, which divides $b_{Z_n}(s)$ by (\ref{eq:divinduction}). This shows that $-3,-5,\cdots,-2n+1$ are roots of $b_d(s)$, and since $b_d(s)$ divides $P_d(s)$, it follows from (\ref{eq:skewPds}) that the only other possible root is $-2n-1$. Using \cite[Thm.~5.5]{raicu-weyman-witt} and Corollary~\ref{cor:rootfromloccoh} with $r=2n+1$ and $I$ being the ideal generated by the $d_i$'s, it follows that $-2n-1$ is indeed a root of $b_d(s)$, hence (\ref{eq:bfunskew}) holds.
\end{proof}

\begin{remark}
The method described in Remark \ref{rem:quiv} can be used in this case as well. Using the decomposition (\ref{eq:skewCauchy}) and the Littlewood-Richardson rule, we see that $d_i^* \cdot S_{((\a+1)^{2n})}\bb{C}^{2n+1} \subset S_{(\a^{2n})}\bb{C}^{2n+1}$ for $\a\in\bb{Z}_{\geq 0}$. Moreover, under the action of diagonal matrices the weights of $d_1,\cdots, d_{2n+1}$ are linearly independent. Hence the tuple $d=(d_1,d_2,\cdots,d_{2n+1})$ has a $b$-function of several variables, and as in the proof of \cite[Theorem 5.1]{bub1} we obtain the formulas
$$d_i^* \cdot d_i \cdot d^{\ul s} = (s_i+1)(s+3)(s+5)\cdots (s+2n-1)\cdot d^{\ul s}.$$
Together with the analogue of Lemma~\ref{lem:equivdefaps} below, this gives the Bernstein-Sato polynomial of the ideal.
\end{remark}

\section{Bernstein--Sato polynomials for maximal minors}\label{sec:bsatomaxlminors}

In this section we generalize Theorem~\ref{thm:n+1byn} to arbitrary $m\times n$ matrices. We use the notation from Sections~\ref{subsec:invDmods} and~\ref{sec:almostsquare}: in particular $d=(d_K)_{K\in{[m]\choose n}}$ is the tuple of maximal minors as in~(\ref{eq:defdK}).

\begin{theorem}\label{thm:bfunmaxlminors}
 The Bernstein--Sato polynomial of the tuple of maximal minors of the generic $m\times n$ matrix~is
\[b_d(s)=\prod_{i=m-n+1}^m (s+i).\]
\end{theorem}

We know by Proposition~\ref{prop:DfgivesPfs} and Theorem~\ref{thm:Pds} that $b_d(s)$ divides $\prod_{i=m-n+1}^m (s+i)$. By induction, we also know from (\ref{eq:divinduction}) that $b_d(s)$ is divisible by $\prod_{i=m-n+1}^{m-1} (s+i)$, so we would be done if we can show that $-m$ is a root of $b_d(s)$. This would follow from Proposition~\ref{prop:rootbfun} if we could prove the following:

\begin{conjecture}
 If we associate as in Section~\ref{subsec:lowerbounds} the $\D$-modules $F_{\a}$, $\a\in\bb{Z}$, to the tuple $d$ of maximal minors of the generic $m\times n$ matrix, then there exists a strict inclusion $F_{-m+1}\subsetneq F_{-m}$.
\end{conjecture}

We weren't able to verify this conjecture when $m>n+1$, so we take a different approach. We consider the $(1+n\cdot(m-n))$-tuple
\[p=(p_0,p_{ij})\in S^{1+n\cdot(m-n)},\]
as in (\ref{eq:defpij}) and associate to $p_0$ a variable $s_0$, and to each $p_{ij}$ a variable $s_{ij}$. We write $\ul{s}=(s_0,s_{ij})$ and consider $B_{\ul s}^p$ as defined in (\ref{eq:defBs}). Inside $B_{\ul s}^p$, we consider the $\bb{C}[\ul{s}]$-submodule
\begin{equation}\label{eq:defAs}
A_{\ul s}^p = \bb{C}[\ul{s}]\cdot\left\{p^{\ul{s}+\hat{c}}=p_0^{s_0+c_0}\cdot\prod_{i\in[n],j\in[m]\setminus[n]} p_{ij}^{s_{ij}+c_{ij}}:\hat{c}=(c_0,c_{ij})\in\bb{Z}^{1+n\cdot(m-n)}\right\}.
\end{equation}
A more invariant way of describing $A_{\ul s}^p$ follows from the discussion in Section~\ref{subsec:invDmods}:
\begin{equation}\label{eq:As=invBs}
A_{\ul s}^p\rm{ consists precisely of the }\mf{sl}_n\rm{-invariants inside the }\D_X\rm{-module }B_{\ul s}^p. 
\end{equation}
It follows that $A_{\ul s}^p$ is in fact a $\D_X^{\mf{sl}_n}$-module. Since $\p_K\in\D_X^{\mf{sl}_n}$ for every $K\in{[m]\choose n}$, we can make the following:
\begin{definition}\label{def:defaps}
 We let $s=s_0+\sum_{i,j}s_{ij}$ and define $a_p(s)$ to be the monic polynomial of the lowest degree in $s$ for which $a_p(s)\cdot p^{\ul s}$ belongs to
 \[\bb{C}[\ul s]\cdot\left\{\p_K\cdot p^{\ul{s}+\hat{c}}:K\in{[m]\choose n},|\hat{c}|=1\right\}.\]
\end{definition}

With $P_d(s)$ as computed in Theorem~\ref{thm:Pds} we will prove that
\begin{equation}\label{eq:apsdividesbds}
 a_p(s)\rm{ divides }b_d(s),\rm{ and}
\end{equation}
\begin{equation}\label{eq:Pdsdividesaps}
 P_d(s)\rm{ divides }a_p(s).
\end{equation}
Combining (\ref{eq:apsdividesbds}) with (\ref{eq:Pdsdividesaps}), and with the fact that $b_d(s)$ divides $P_d(s)$, concludes the proof of Theorem~\ref{thm:bfunmaxlminors}.

It follows from (\ref{eq:slninvariantsS_p}) that the elements $p^{\ul{s}+\hat{c}}$ in (\ref{eq:defAs}) in fact give a basis of $A_{\ul s}^p$ as a $\bb{C}[\ul{s}]$-module. We have
\[A_{\ul s}^p=\bigoplus_{\a\in\bb{Z}} A_{\ul s}^p(\a)\]
which we can think of as a weight space decomposition, where
\begin{equation}\label{eq:defAsalpha}
 A_{\ul s}^p(\a)=\bb{C}[\ul{s}]\cdot\left\{p^{\ul{s}+\hat{c}}:|\hat{c}|=\a\right\}
\end{equation}
is the set of elements in $A_{\ul s}^p$ on which $g\in\mf{gl}_n$ acts by multiplication by $\rm{tr}(g)\cdot(s+\a)$, and in particular each $A_{\ul s}^p(\a)$ is preserved by $\D_X^{\mf{gl}_n}$. Using (\ref{eq:rel}) we obtain that multiplication by $d_K$ sends $A_{\ul s}^p(\a)$ into $A_{\ul s}^p(\a+1)$. Since $d_K\cdot\p_K\in\D_X^{\mf{gl}_n}$, it then follows that multiplication by $\p_K$ sends $A_{\ul s}^p(\a+1)$ into $A_{\ul s}^p(\a)$. We obtain:

\begin{lemma}\label{lem:equivdefaps}
 The polynomial $a_p(s)$ is the monic polynomial of lowest degree for which there exist a finite collection of tuples $\hat{c}\in\bb{Z}^{1+n\cdot(m-n)}$ with $|\hat{c}|>0$ and corresponding operators $Q_{\hat{c}}\in\D_X[\ul s]$ such that
\begin{equation}\label{eq:Qcopstoaps}
\sum_{\hat{c}} Q_{\hat{c}}\cdot p^{\ul{s}+\hat{c}}=a_p(s)\cdot p^{\ul s}. 
\end{equation}
\end{lemma}

\begin{proof} Using the fact that $p^{\ul{s}+\hat{c}}$ and $a_p(s)\cdot p^{\ul s}$ are $\mf{sl}_n$-invariants, we may assume that $Q_{\hat{c}}\in\D_X^{\mf{sl}_n}[\ul{s}]$. Since every element in $\D_X^{\mf{sl}_n}$ can be expressed as a linear combination of products $Q_1\cdot Q_2\cdot Q_3$, where $Q_1$ is a product of $\p_K$'s, $Q_2$ is a product of $d_K$'s, and $Q_3\in\D_X^{\mf{gl}_n}$, the conclusion follows from the observation that $\D_X^{\mf{gl}_n}$ preserves each weight space, $d_K$ increases the weight by one, while $\p_K$ decreases the weight by one. 
\end{proof}

We are now ready to prove that $a_p(s)$ divides $b_d(s)$:
\begin{proof}[Proof of (\ref{eq:apsdividesbds})]
 Using (\ref{eq:genbfseqn}) with $\ul{s}=(s_K)_{K\in{[m]\choose n}}$ we can find a finite collection of tuples $\hat{c}\in\bb{Z}^{[m]\choose n}$ with $|\hat{c}|=1$, and corresponding operators $P_{\hat{c}}\in\D_X[\ul{s}]$ such that we have an equality inside $B_{\ul{s}}^d$:
\begin{equation}\label{eq:opforbds}
 \sum_{\hat{c}}\hat{c}_{\ul s}\cdot P_{\hat{c}}\cdot d^{\ul{s}+\hat{c}}=b_d(s)\cdot d^{\ul s}. 
\end{equation}
Note that by (\ref{eq:defchat}), setting $s_K=0$ makes $\hat{c}_{\ul s}=0$ whenever $\hat{c}$ is such that $c_K<0$. We apply to (\ref{eq:opforbds}) the specialization
\begin{equation}\label{eq:specialization}
s_K=0\rm{ whenever }|K\cap [n]|\leq n-2,\ s_{[n]}=s_0,\rm{ and }s_{[n]\setminus\{i\}\cup\{j\}}=s_{ij}\rm{ for }i\in[n],j\in[m]\setminus[n].
\end{equation}
We then use the equalities $p_0=d_{[n]}$, $p_{ij}=d_{[n]\backslash\{i\}\cup\{j\}}$ and (\ref{eq:rel}), and regroup the terms to obtain (with an abuse of notation) a finite collection of tuples $\hat{c}=(c_0,c_{ij})\in\bb{Z}^{1+n\cdot(m-n)}$ with $|\hat{c}|=1$, and corresponding operators $Q_{\hat{c}}\in\D_X[\ul{s}]$, where $\ul{s}$ denotes now the tuple of variables $(s_0,s_{ij})$, such that the following equality holds in~$B_{\ul{s}}^p$:
\[\sum_{\hat{c}} Q_{\hat{c}}\cdot p^{\ul{s}+\hat{c}}=b_d(s)\cdot p^{\ul s}.\]
Using Lemma~\ref{lem:equivdefaps} it follows that $a_p(s)$ divides $b_d(s)$ as desired.
\end{proof}

We conclude by proving (\ref{eq:Pdsdividesaps}), but before we establish a preliminary result. For $|\hat{c}|=1$ we observe that $p^{\ul{s}+\hat{c}}\in A_{\ul{s}}^p(1)$, thus $\p_K\cdot p^{\ul{s}+\hat{c}}$ can be expressed as a $\bb{C}[\ul{s}]$-linear combination of the basis elements of $A_{\ul{s}}^p(0)$. We define $Q_{K,\hat{c}}\in\bb{C}[\ul{s}]$ to be the coefficient of $p^{\ul{s}}$ in this expression, and write $\hat{e}=(1,0^{n\cdot(m-n)})$.

\begin{lemma}\label{lem:QKc=0}
 Write $Q_{K,\hat{c}}^0\in\bb{C}[s_0]$ for the result of the specialization $s_{ij}=-1$ for all $i\in[n],j\in[m]\setminus[n]$, applied to $Q_{K,\hat{c}}$. We have that $Q_{K,\hat{c}}^0=0$ unless $K=[n]$ and $\hat{c}=\hat{e}$.
\end{lemma}

\begin{proof} Since the specialization map commutes with the action of $\D_X$, we have that
 \[Q_{K,\hat{c}}^0\rm{ is the coefficient of }\frac{p_0^{s_0}}{\prod_{i,j} p_{ij}}\rm{ inside }\p_K\cdot\left(p_0^{s_0+c_0}\cdot\prod_{i,j} p_{ij}^{c_{ij}-1}\right).\]

Suppose first that $\hat{c}$ is a tuple with some entry $c_{i_0j_0}\geq 1$: we show that for any $K$, $Q_{K,\hat{c}}^0=0$. To see this, note that applying any sequence of partial derivatives to $p_0^{s_0+c_0}\cdot\prod_{i,j}p_{ij}^{c_{ij}-1}$ won't turn the exponent of $p_{i_0j_0}$ negative. Since $\p_K\in\D_X^{\mf{sl}_n}$, we may then assume that
\begin{equation}\label{eq:parKofproduct}
\p_K\cdot\left(p_0^{s_0+c_0}\cdot\prod_{i,j} p_{ij}^{c_{ij}-1}\right)=p_0^{s_0+d_0}\cdot\prod_{i,j}p_{ij}^{d_{ij}}\cdot F, 
\end{equation}
where $d_0,d_{ij}\in\bb{Z}$, $d_{i_0j_0}=0$, and $F\in S^{\mf{sl}_n}[s_0]$ is a polynomial in $s_0$ whose coefficients are $\mf{sl}_n$-invariant. Since $S^{\mf{sl}_n}$ is generated by the maximal minors $d_K$, we can apply (\ref{eq:rel}) to rewrite the right hand side of (\ref{eq:parKofproduct}) as a $\bb{C}[s_0]$-linear combination of $p_0^{s_0+e_0}\cdot\prod_{i,j}p_{ij}^{e_{ij}}$ where $e_0,e_{ij}\in\bb{Z}$ and $e_{i_0j_0}\geq 0$. We conclude that $Q_{K,\hat{c}}^0=0$.

From now on we assume that $\hat{c}$ is has all $c_{ij}\leq 0$. Since $|\hat{c}|=1$, we must have $c_0\geq 1$. We look at weights under the action of the subalgebra
\[\left\{T_t=\begin{pmatrix}
t\cdot I_n & 0 \\
0 & 0
\end{pmatrix}\,:\, t\in \bb{C}\right\}\subset \mf{gl}_m,\]
and note that
\[T_t\cdot\left(p_0^{s_0+c_0}\cdot\prod_{i,j} p_{ij}^{c_{ij}-1}\right) = t\cdot\left((s_0+c_0)\cdot n+(n-1)\sum_{i,j}(c_{ij}-1)\right)\cdot\left(p_0^{s_0+c_0}\cdot\prod_{i,j} p_{ij}^{c_{ij}-1}\right),\]
\[T_t\bul\p_K = -t\cdot |K\cap[n]|\cdot\p_K,\rm{ using notation }(\ref{eq:bulactiong}),\rm{ and}\]
\[T_t\cdot\left(\frac{p_0^{s_0}}{\prod_{i,j} p_{ij}}\right) = t\cdot\left(s_0\cdot n+(n-1)\sum_{i,j}(-1)\right)\cdot\left(\frac{p_0^{s_0}}{\prod_{i,j} p_{ij}}\right).\]
It follows that $Q_{K,\hat{c}}^0$ can be non-zero only when
\[(s_0+c_0)\cdot n+(n-1)\sum_{i,j}(c_{ij}-1)-|K\cap[n]|=s_0\cdot n+(n-1)\sum_{i,j}(-1),\]
which using the fact that $c_0+\sum_{i,j}c_{ij}=1$ is equivalent to $c_0+(n-1)=|K\cap[n]|$. Since $c_0\geq 1$ this equality can only hold when $c_0=1$ (which then forces all $c_{ij}=0$), and $K=[n]$.
\end{proof}

\begin{proof}[Proof of (\ref{eq:Pdsdividesaps})]
Using Definition~\ref{def:defaps}, we can find finitely many tuples $\hat{c}\in\bb{Z}^{1+n\cdot(m-n)}$ with $|\hat{c}|=1$, and polynomials $P_{K,\hat{c}}\in\bb{C}[\ul{s}]$ for $K\in{[m]\choose n}$ such that
\begin{equation}\label{eq:PQaps}
\sum_{K,\hat{c}} P_{K,\hat{c}}\cdot\p_K\cdot p^{\ul{s}+\hat{c}}=a_p(s)\cdot p^{\ul{s}}. 
\end{equation}
Using the definition of $Q_{K,\hat{c}}$, we obtain
\[\sum_{K,\hat{c}} P_{K,\hat{c}}\cdot Q_{K,\hat{c}}=a_p(s).\]
Applying the specialization $s_{ij}=-1$ for all $i\in [n],j\in [m]\backslash[n]$, it follows from Lemma~\ref{lem:QKc=0} that
\[P_{[n],\hat{e}}^0\cdot Q_{[n],\hat{e}}^0=\sum_{K,\hat{c}} P_{K,\hat{c}}^0\cdot Q_{K,\hat{c}}^0=a_p(s_0-n\cdot(m-n)),\]
where $P_{K,\hat{c}}^0\in\bb{C}[s_0]$ is (just as $Q_{K,\hat{c}}^0$) the specialization of $P_{K,\hat{c}}$. We will show that $Q_{[n],\hat{e}}^0=P_d(s_0-n\cdot(m-n))$, from which it follows that $P_d(s_0-n\cdot(m-n))$ divides $a_p(s_0-n\cdot(m-n))$. Making the change of variable $s=s_0-n\cdot(m-n)$ proves that $P_d(s)$ divides $a_p(s)$, as desired.

To see that $Q_{[n],\hat{e}}^0=P_d(s_0-n\cdot(m-n))$, we consider the action of $D_d$ on $p^{\ul{s}}$: using (\ref{eq:Pds}), Theorem~\ref{thm:Pds}, and applying the specialization (\ref{eq:specialization}) as before, we obtain
\[\sum_{K\in{[m]\choose n}}\p_K\cdot d_K\cdot p^{\ul s}=D_d\cdot p^{\ul s}=P_d(s)\cdot p^{\ul s}.\]
Using (\ref{eq:rel}), we can rewrite the above equality as
\[\p_{[n]}\cdot p^{\ul{s}+\hat{e}}+\sum_{\substack{K\neq[n]\\\hat{c}\rm{ with }|\hat{c}|=1}}R_{K,\hat{c}}\cdot\p_K\cdot p^{\ul{s}+\hat{c}}=P_d(s)\cdot p^{\ul{s}},\]
for some $R_{K,\hat{c}}\in\bb{C}[\ul{s}]$. We now apply the same argument as we did to (\ref{eq:PQaps}): we consider the further specialization $s_{ij}=0$ and use Lemma~\ref{lem:QKc=0} to obtain $Q_{[n],\hat{e}}^0=P_d(s_0-n\cdot(m-n))$, which concludes our~proof.
\end{proof}

\section{The Strong Monodromy Conjecture for maximal minors and sub-maximal Pfaffians}\label{sec:SMC}

Let $X=\bb{C}^N$ and $Y\subset X$ a closed subscheme with defining ideal $I$. Consider a \defi{log resolution} $f:X'\to X$ of the ideal $I$ (or of the pair $(X,Y)$; see for instance \cite[Sec.~9.1.B]{lazarsfeld}), i.e. a proper birational morphism $f:X'\to X$ such that $I\mc{O}_{X'}$ defines an effective Cartier divisor $E$, $f$ induces an isomorphism $f:X'\setminus E\to X\setminus Y$, and the divisor $K_{X'/X}+E$ has simple normal crossings support. Write $E_j$, $j\in\mc{J}$, for the irreducible components of the support of $E$, and express
\[E=\sum_{j\in\mc{J}}a_j E_j,\quad K_{X'/X}=\sum_{j\in\mc{J}}k_j\cdot E_j.\]
The \defi{topological zeta function} of $I$ (or of the pair $(X,Y)$) is defined as \cites{denef-loeser,denef-loeser-jag,veys}
\begin{equation}\label{eq:deftopzeta}
 Z_I(s)=\sum_{\mc{I}\subseteq\mc{J}}\chi(E_{\mc{I}}^{\circ})\cdot\prod_{i\in\mc{I}}\frac{1}{a_i\cdot s+k_i+1},
\end{equation}
where $\chi$ denotes the Euler characteristic and $E_{\mc{I}}^{\circ}=(\bigcap_{i\in\mc{I}}E_i)\setminus(\bigcup_{i\notin\mc{I}}E_i)$. The topological zeta function is independent of the log resolution, and the Strong Monodromy Conjecture asserts that the poles of $Z_I(s)$ are roots of $b_I(s)$, and in an even stronger form that
\begin{equation}\label{eq:SMC}
b_I(s)\cdot Z_I(s)\rm{ is a polynomial.}
\end{equation}
We verify (\ref{eq:SMC}) for maximal minors and sub-maximal Pfaffians as a consequence of Theorems~\ref{thm:bfunskew} and~\ref{thm:bfunmaxlminors}, by taking advantage of the well-studied resolutions given by \defi{complete collineations} in the case of determinantal varieties, and \defi{complete skew forms} in the case of Pfaffian varieties \cites{vainsencher,thaddeus,johnson}.

\subsection{Maximal minors}

Let $m\geq n$ and $X=X_{m,n}$ denote the vector space of $m\times n$ matrices as before. Denote by $Y$ the subvariety of matrices of rank at most $n-1$, and let $I$ be the ideal of maximal minors defining $Y$. It follows from \cite[Cor.~4.5~and~Cor.~4.6]{johnson} that $I$ has a log resolution with $\mc{J}=\{0,\cdots,n-1\}$ and
\[E=\sum_{i=0}^{n-1}(n-i)\cdot E_i,\quad K_{X'/X}=\sum_{i=0}^{n-1}((m-i)(n-i)-1)\cdot E_i.\]
It follows that $k_i+1=(m-i)(n-i)$, and $a_i=n-i$ for $i=0,\cdots,n-1$, and therefore by our Theorem~\ref{thm:bfunmaxlminors} the denominator of every term in (\ref{eq:deftopzeta}) divides $b_I(s)$. This is enough to conclude (\ref{eq:SMC}).

\subsection{Sub-maximal Pfaffians}

Let $X=X_n$ be the vector space of $(2n+1)\times(2n+1)$ skew-symmetric matrices. Denote by $Y$ the subvariety of matrices of rank at most $2(n-1)$ and let $I$ denote the ideal of sub-maximal Pfaffians defining $Y$. As shown below, there is a log resolution of $I$ with $\mc{J}=\{0,\cdots,n-1\}$~and
\begin{equation}\label{eq:EandK}
E=\sum_{i=0}^{n-1}(n-i)\cdot E_i,\quad K_{X'/X}=\sum_{i=0}^{n-1}(2(n-i)^2+(n-i)-1)\cdot E_i. 
\end{equation}
It follows that $(k_i+1)/a_i=2(n-i)+1$ for $i=0,\cdots,n-1$, and thus our Theorem~\ref{thm:bfunskew} implies (\ref{eq:SMC}).

We sketch the construction of the log resolution, based on the strategy in \cite[Chapter~4]{johnson}: this is perhaps well-known, but we weren't able to locate (\ref{eq:EandK}) explicitly in the literature. We write $Y_i\subset X$ for the subvariety of $(2n+1)\times(2n+1)$ skew-symmetric matrices of rank at most $2i$. We define the sequence of transformations $\pi_i:X^{i+1}\to X^i$, $f_i=\pi_0\circ\pi_1\circ\cdots\circ\pi_i:X^{i+1}\to X^0$, where $X^0=X$, $X^1$ is the blow-up of $X^0$ at $Y_0$, and in general $X^{i+1}$ is the blow-up of $X^i$ at the strict transform $Y^i$ of $Y_i$ along $f_{i-1}$. The desired log resolution is obtained by letting $X'=X^n$ and $f=f_{n-1}:X'\to X$. Each $Y^i$ is smooth (as we'll see shortly), so the same is true about the exceptional divisor $E_i$ of the blow-up $\pi_i$. We abuse notation and write $E_i$ also for each of its transforms along the blow-ups $\pi_{i+1},\cdots,\pi_{n-1}$. It follows from the construction below that the $E_i$'s are defined locally by the vanishing of distinct coordinate functions, so $f:X'\to X$ is indeed a log resolution.

We show by induction on $i=n,n-1,\cdots$ that $X^{n-i}$ admits an affine open cover where each open set $V$ in the cover has a filtration $V=V_i\supset V_{i-1}\supset\cdots\supset V_0$, isomorphic to
\begin{equation}\label{eq:filtration}
(Y_i^i\supset Y_{i-1}^i\supset\cdots\supset Y_0^i)\times\bb{C}^{4i+3}\times\cdots\times\bb{C}^{4(n-1)-1}\times\bb{C}^{4n-1}, 
\end{equation}
where $Y_i^n=Y_i$ and more generally
\[Y_j^i\rm{ is the variety of }(2i+1)\times(2i+1)\rm{ matrices of rank at most }2j.\]
The key property of the filtration (\ref{eq:filtration}) is that for each $j=0,\cdots,i$, $V_j$ is obtained by intersecting $V$ with the strict transform of $Y_{n-i+j}$ along $f_{n-i-1}$. In particular $V_0=V\cap Y^{n-i}$ is (on the affine patch $V$) the center of blow-up for $\pi_i$. Since $Y_0^0$ is just a point, $V_0$ is an affine space and hence smooth.

When $i=n$, $X^{n-i}=X$, so we can take $V=X$ and (\ref{eq:filtration}) to be the filtration $X=Y_n\supset Y_{n-1}\supset\cdots\supset Y_0$. We discuss the first blow-up ($i=n-1$) and the associated filtration, while for $i<n-1$ the conclusion follows from an easy iteration of our argument. We write $x_{ij}$ (resp. $y_{ij}$), $1\leq i<j\leq 2n+1$ for the coordinate functions on $X$ (resp. on $\bb{P}X$, the projectivization of $X$). $X^1$ is defined inside $X\times\bb{P}X$ by the equations $x_{ij}y_{kl}=x_{kl}y_{ij}$, and we choose $V\subset X^1$ to be the affine patch where $y_{12}\neq 0$ (similar reasoning applies on each of the affine patches $y_{ij}\neq 0$). The coordinate functions on $V$ are $t_0=x_{12}$ and $u_{ij}=y_{ij}/y_{12}$ for $(i,j)\neq (1,2)$. Setting $u_{12}=1$, we get that the map $\pi_0:V\to X^0$ corresponds to a ring homomorphism
\[\pi_0^*:\bb{C}[x_{ij}]\lra\bb{C}[t_0,u_{ij}]\rm{ given by }x_{ij}\mapsto t_0\cdot u_{ij},\]
and $E_0\cap V$ is defined by the equation $t_0=0$. With the usual conventions $u_{ji}=-u_{ij}$, $u_{ii}=0$, we write $M_{ij}=\Pf_{\{1,2,i+2,j+2\}}$ for the Pfaffian of the $4\times 4$ principal skew-symmetric submatrix of $(u_{ij})$ obtained by selecting the rows and columns of $u$ indexed by $1,2,i+2$ and $j+2$, $1\leq i,j\leq 2n-1$. Using the calculation in the proof of Lemma~\ref{lem:localizeskewsym}, we obtain that $\{M_{ij}: 1\leq i<j\leq 2n-1\}\cup\{t_0\}\cup\{u_{1i}$, $u_{2i}:i=3,\cdots,2n+1\}$ is a system of coordinate functions on $V$, and moreover
\begin{equation}\label{eq:transformIp}
\pi_0^*(I_{p+1}(x_{ij}))=t_0^{p+1}\cdot I_p(M_{ij}),\rm{ for }p=1,\cdots,n,
\end{equation}
where $I_p(a_{ij})$ denotes the ideal generated by the $2p\times 2p$ Pfaffians of the skew-symmetric matrix $(a_{ij})$. Thinking of $\{t_0\}\cup\{u_{1i},u_{2i}:i=3,\cdots,2n+1\}$, as the coordinate functions on $\bb{C}^{4n-1}$, and of $\{M_{ij}\}$ as the coordinate functions on $X_{n-1}=Y_{n-1}^{n-1}$, we identify $Y^{n-1}_{p-1}$ with the zero locus of $I_p(M_{ij})$ for $p=1,\cdots,n$, and note that by (\ref{eq:transformIp}) it is the strict transform of $Y_p$ which is the variety defined by $I_{p+1}(x_{ij})$. This yields the filtration (\ref{eq:filtration}) for $i=n-1$. By letting $p=n-1$ in (\ref{eq:transformIp}) and noting that $I=I_n(x_{ij})$, we obtain that the inverse image $\pi_0^{-1}(I)=I\mc{O}_{X^1}$ vanishes with multiplicity $n$ along $E_0$. Iterating this, we obtain the formula (\ref{eq:EandK}) for the exceptional divisor $E$. Pulling back the standard volume form $dx=dx_{12}\wedge\cdots\wedge dx_{n-1,n}$ on $X$ along $\pi_0$, we obtain (on the affine patch $V$)
\[\pi_0^*(dx)=t_0^{2n^2+n-1}\cdot dt_0\wedge du_{13}\wedge\cdots\wedge du_{n-1,n},\]
which vanishes with multiplicity $2n^2+n-1$ along $E_0$. Iterating this, we obtain formula (\ref{eq:EandK}) for $K_{X'/X}$.

\section*{Acknowledgments} 
We are grateful to Nero Budur for many interesting conversations and helpful suggestions. Experiments with the computer algebra software Macaulay2 \cite{M2} have provided numerous valuable insights. Raicu acknowledges the support of the National Science Foundation under grant DMS-1458715. Walther acknowledges the support of the National Science Foundation under grant DMS-1401392. Weyman acknowledges the support of the Alexander von Humboldt Foundation, and of the National Science Foundation under grant DMS-1400740.

	%%%%%%%%%%%%%%%%%%%%%%%%%%%%%%%%%%%%%%%%%%%%%%%%%%%%%%%%%%%%%%%%%%%%%%%%
	%%%%%%%%%%%%%%%   		Bibliography				%%%%%%%%%%%%%%%%%%%%
	%%%%%%%%%%%%%%%%%%%%%%%%%%%%%%%%%%%%%%%%%%%%%%%%%%%%%%%%%%%%%%%%%%%%%%%%

\end{document}